\documentclass{article}   	
               						
\usepackage {amsmath,float}
\usepackage {amssymb}
\usepackage {bm,url}
\usepackage {graphicx}
\usepackage {graphpap}
\usepackage {longtable}
\usepackage {multirow}
\usepackage {amsthm}

\newenvironment{newreferences}
               {\section*{References}
                \begin{list}{}{\setlength{\itemsep}{0pt}
                               \setlength{\parsep}{0pt}
                               \setlength{\labelwidth}{0pt}
                               \setlength{\leftmargin}{12pt}
                               \setlength{\labelsep}{0pt}}
                \setlength{\itemindent}{-12pt}
               }{\end{list}
               }

\allowdisplaybreaks
\newcommand{\Z}{{\bf Z}}
\theoremstyle{plain}
\newtheorem{theorem}{Theorem}
\newtheorem{lemma}[theorem]{Lemma}

\theoremstyle{remark}
\usepackage {tikz}
\usetikzlibrary {decorations.markings}

\usepackage {varioref}

\title{Spatial Parrondo games \\ and an interacting particle system}
\author{Sung Chan Choi\thanks{Department of Mathematics, University of Utah, 155 S. 1400 E., Salt Lake City, UT 84112, USA. e-mail: choi@math.utah.edu}}
\date{}							

\begin{document}
\maketitle

\begin{abstract}
Parrondo games with spatial dependence were introduced by Toral (2001) and have been studied extensively.  In Toral's model $N$ players are arranged in a circle.  The players play either game $A$ or game $B$. In game $A$, a randomly chosen player wins or loses one unit according to the toss of a fair coin.  In game $B$, which depends on parameters $p_0,p_1,p_2\in[0,1]$, a randomly chosen player, player $x$ say, wins or loses one unit according to the toss of a $p_m$-coin, where $m\in\{0,1,2\}$ is the number of nearest neighbors of player $x$ who won their most recent game.  In this paper, we replace game $A$ by a spatially dependent game, which we call game $A'$, introduced by Xie et al.~(2011).  In game $A'$, two nearest neighbors are chosen at random, and one pays one unit to the other based on the toss of a fair coin.  Game $A'$ is fair, so we say that the Parrondo effect occurs if game $B$ is losing or fair and the game $C'$, determined by a random or periodic sequence of games $A'$ and $B$, is winning.  Here we give sufficient conditions for convergence as $N\to\infty$ of the mean profit per game played from game $C'$.  This requires ergodicity of an associated interacting particle system (not necessarily a spin system), for which sufficient conditions are found using the basic inequality.
\end{abstract}

\section{Introduction}\label{sec-intro}

Toral (2001) introduced what he called \textit{cooperative} Parrondo games with spatial dependence.  (We prefer the term \textit{spatially dependent} Parrondo games so as to avoid conflict with the field of cooperative game theory.)  The games depend on an integer parameter $N\ge3$, the number of players, and four probability parameters, $p_0,p_1,p_2,p_3$.  (This is a slight generalization of the model described in the abstract.) The players are arranged in a circle and labeled from 1 to $N$, so that players 1 and $N$ are adjacent.  At each turn, a player is chosen at random to play.  Suppose player $x$ is chosen.  In game $A$, he tosses a fair coin.  In game $B$, he tosses a $p_m$-coin (i.e., a coin whose probability of heads is $p_m$), where $m\in\{0,1,2,3\}$ depends on the winning or losing status of his two nearest neighbors.  A player's status as winner (1) or loser (0) is decided by the result of his most recent game.  Specifically,
$$
m=\begin{cases}0&\text{if $x-1$ and $x+1$ are both losers,}\\
               1&\text{if $x-1$ is a loser and $x+1$ is a winner,}\\
               2&\text{if $x-1$ is a winner and $x+1$ is a loser,}\\
               3&\text{if $x-1$ and $x+1$ are both winners,}\end{cases}
$$
where $N+1:=1$ and $0:=N$ because of the circular arrangement of players.
Player $x$ wins one unit with heads and loses one unit with tails.

Game $A$ is fair, so we say that the Parrondo effect occurs if game $B$ is losing or fair and the game $C$, determined by a random or periodic sequence of games $A$ and $B$, is winning.

These games have been studied by Mihailovi\'c and Rajkovi\'c (2003) and Ethier and Lee (2012a,b, 2013a,b), and a strong law of large numbers and a central limit theorem were obtained.  In particular, the asymptotic cumulative profits per turn exist and are the means in the SLLN.  Further, it seems clear that these means converges as $N\to\infty$.  This has been proved under certain conditions (see Ethier and Lee, 2013a).

Notice that Toral's (2001) game $A$ is not spatially dependent (the rules of the game do not depend on the spatial structure of the players). Xie et al.~(2011) proposed a modification of game $A$ that \textit{is} spatially dependent as well as being a fair game.  To distinguish, we call that game $A'$.  As before, the games depend on an integer parameter $N\ge3$, the number of players, and four probability parameters, $p_0,p_1,p_2,p_3$.  The players are arranged in a circle and labeled from $1$ to $N$, so that players $1$ and $N$ are adjacent.  At each turn, a player is chosen at random to play.  Suppose player $x$ is chosen.  In game $A'$, he chooses one of his two nearest neighbors at random and competes with that neighbor by tossing a fair coin.  The result is a transfer of one unit from one of the players to the other, hence the wealth of the set of $N$ players is unchanged.  Game $B$ is as before.  Player $x$ wins one unit with heads and loses one unit with tails. 

Game $A'$ is fair, so we say that the Parrondo effect occurs if game $B$ is losing or fair and the game $C'$, determined by a random or periodic sequence of games $A'$ and $B$, is winning.

These games were studied by Xie et al.~(2011), Li et al.~(2014), and Ethier and Lee (2015).  Only the case of random sequences was treated, and convergence of the means has not yet been addressed.  Our aim in this paper, together with Choi (2020), is to fill in these gaps in the literature.

The Markov chain formalized by Mihailovi\'c and Rajkovi\'c (2003) keeps track of the status (loser or winner, 0 or 1) of each of the $N\ge3$ players of game $B$. Its state space is the product space 
$$
\{\eta=(\eta(1),\eta(2),\ldots,\eta(N)): \eta(x)\in\{0,1\}{\rm\ for\ }x=1,\ldots,N\}=\{0,1\}^N
$$
with $2^N$ states.  Let $m_x(\eta):=2\eta(x-1)+\eta(x+1)\in\{0,1,2,3\}$.  Of course $\eta(0):=\eta(N)$ and $\eta(N+1):=\eta(1)$ because of the circular arrangement of players.  Also, let $\eta_x$ be the element of $\{0,1\}^N$ equal to $\eta$ except at the $x$th coordinate.  For example, $\eta_1:=(1-\eta(1),\eta(2),\eta(3),\ldots,\eta(N))$.

The one-step transition matrix $\bm P_B$ for this Markov chain depends not only on $N$ but on four parameters, $p_0,p_1,p_2,p_3\in[0,1]$.  It has the form
\begin{equation*}
P_B(\eta,\eta_x):=\begin{cases}N^{-1}p_{m_x(\eta)}&\text{if $\eta(x)=0$,}\\N^{-1}q_{m_x(\eta)}&\text{if $\eta(x)=1$,}\end{cases}\qquad x=1,\ldots,N,\;\eta\in \{0,1\}^N,
\end{equation*}
and 
\begin{equation*}
P_B(\eta,\eta):=N^{-1}\bigg(\sum_{x:\eta(x)=0}q_{m_x(\eta)}+\sum_{x:\eta(x)=1}p_{m_x(\eta)}\bigg),\qquad \eta\in\{0,1\}^N,
\end{equation*}
where $q_m:=1-p_m$ for $m=0,1,2,3$ and empty sums are 0.  The Markov chain is irreducible and aperiodic if $0<p_m<1$ for $m=0,1,2,3$.  Under slightly weaker assumptions (see Ethier and Lee 2013a), the Markov chain is ergodic, which suffices.  For example, if $p_0$ is arbitrary and $0<p_m<1$ for $m=1,2,3$, or if $0<p_m<1$ for $m=0,1,2$ and $p_3$ is arbitrary, then ergodicity holds.

A Markov chain in the same state space corresponds to game $A'$.  As previously mentioned, if game $A'$ is played, the profit to the set of $N$ players is 0, since game $A'$ simply redistributes capital among the players.  The transition probabilities require some new notation.  Starting from state $\eta$, let $\eta^{x,x\pm1,\pm1}\in\{0,1\}^N$ represent the players' status after player $x$ plays player $x\pm1$ and wins (1) or loses ($-1$).  Of course player 0 is player $N$ and player $N+1$ is player 1.  For example, $\eta^{1,2,-1}=(0,1,\eta(3),\ldots,\eta(N))$ (player 1 competes against player 2 and loses, leaving player 1 a loser and player 2 a winner, regardless of their previous status).  The transition probabilities have the form
\begin{align*}
&P_{A'}(\eta,\zeta)\\
&:=\frac{1}{4N}\sum_{x=1}^N[\delta(\eta^{x,x-1,-1},\zeta)+\delta(\eta^{x,x-1,1},\zeta)+\delta(\eta^{x,x+1,-1},\zeta)+\delta(\eta^{x,x+1,1},\zeta)],
\end{align*}
where $\delta(\eta,\zeta)=1$ if $\eta=\zeta$ and $=0$ otherwise.

Next, we want to regard the players, originally labeled from 1 to $N$, as labeled from $l_N$ to $r_N$, where
\begin{equation}\label{relabel}
l_N: = \begin{cases}
-(N-1)/2 &\text{if $N$ is odd,}\\
-N/2 &\text{if $N$ is even,}
\end{cases} \quad
r_N: = \begin{cases}
(N-1)/2 &\text{if $N$ is odd,}\\
N/2-1 &\text{if $N$ is even.}
\end{cases}
\end{equation}
Then we can speed up time, playing $N$ games per unit of time, and our process is described in the limit as $N\to\infty$ by an interacting particle system in the state space $\Sigma:=\{0,1\}^\Z$.  The details of this limit operation are postponed to Section \ref{sec-convergence}.  Initially, our concern is with the ergodicity of the limiting interacting particle system.

The generator $\Omega_{A'}$ of the interacting particle system corresponding to game $A'$ can be described as follows.  For $\eta\in\Sigma:=\{0,1\}^\Z$ and $x\in\Z$, define $\eta_x$ and $_x\eta_{x+1}$ in $\Sigma$ by
\begin{equation*}
\eta_x(y):=\begin{cases}1-\eta(x)&\text{if $y=x$,}\\ \eta(y)&\text{otherwise,}\end{cases}\quad\text{and}\quad
_x\eta_{x+1}(y):=\begin{cases}\eta(x+1)&\text{if $y=x$,}\\ \eta(x)&\text{if $y=x+1$,}\\ \eta(y)&\text{otherwise.}\end{cases}
\end{equation*}
Then
\begin{equation*}
(\Omega_{A'} f)(\eta):=\sum_x c'(x,\eta)[f(\eta_x)-f(\eta)]+\frac12\sum_x[f(_x\eta_{x+1})-f(\eta)]
\end{equation*}
for $f\in C(\Sigma)$ depending on only finitely many coordinates, where 
\begin{equation}\label{c'-def}
c'(x,\eta):=\frac{1}{2}[\bm{1}_{\{\eta(x)=\eta(x+1)\}}+\bm{1}_{\{\eta(x)=\eta(x-1)\}}].
\end{equation}
Note that the generator $\Omega_{A'}$ is the sum of a spin system generator and an exclusion process generator.

The generator $\Omega_B$ of the interacting particle system corresponding to game $B$ can be described as follows.  For each $x\in\Z$, define $\eta_x$ as above.  Then, given parameters $p_0,p_1,p_2,p_3\in[0,1]$,
\begin{equation*}
(\Omega_{B} f)(\eta):=\sum_x c(x,\eta)[f(\eta_x)-f(\eta)]
\end{equation*} 
for $f\in C(\Sigma)$ depending on only finitely many coordinates, where 
\begin{equation}\label{c-def}
c(x,\eta):=\begin{cases}p_{m_{x}(\eta)}&\text{if $\eta(x)=0$,}\\ q_{m_{x}(\eta)}&\text{if $\eta(x)=1,$}\end{cases}
\end{equation}
$q_m:=1-p_m$ for $m=0,1,2,3$, and $m_x(\eta):=2\eta(x-1)+\eta(x+1)\in\{0,1,2,3\}$.

Finally, the generator $\Omega_{C'}$ of the interacting particle system corresponding to game $C':=\gamma A'+(1-\gamma)B$ (denoting the random mixture of games $A'$ and $B$, i.e., the game in which a $\gamma$-coin determines whether game $A'$ or game $B$ is played) is
$$
\Omega_{C'}:=\gamma\Omega_{A'}+(1-\gamma)\Omega_B.
$$
The generator $\Omega_{C'}$, just like $\Omega_{A'}$, is the sum of a spin system generator and an exclusion process generator.

\section{Ergodicity}\label{sec-basic}

First, $\Omega_B$ is a spin system generator, so the basic inequality yields a sufficient condition for ergodicity (Liggett, 1985, Eq.~(III.0.6)):
\begin{equation}\label{erg-basic-B}
\sup_{x \in \Z}\sum_{u:u\neq x}\sup_{\eta \in \Sigma} |c(x,\eta)-c(x,\eta_u)| < \inf_{x \in \Z,\; \eta \in \Sigma}[c(x,\eta)+c(x,\eta_{x})],
\end{equation}
where $c(x,\eta)$ is as in \eqref{c-def}.

Next, $\Omega_{C'}$, given by   
\begin{align}\label{Omega-C'}
(\Omega_{C'} f)(\eta)&:=\gamma(\Omega_{A'}f)(\eta)+(1-\gamma)(\Omega_Bf)(\eta)\nonumber\\
&\;=\gamma\sum_x c'(x,\eta)[f(\eta_x)-f(\eta)]+\frac{\gamma}{2}\sum_x[f(_x\eta_{x+1})-f(\eta)]\nonumber\\
&\qquad\qquad+(1-\gamma)\sum_xc(x,\eta)[f(\eta_x)-f(\eta)],
\end{align} 
is not a spin system generator, so \eqref{erg-basic-B} does not apply.  Instead, we apply a more general form of the basic inequality. It assumes that the generator has the form
\begin{equation}\label{Omega-general}
(\Omega f)(\eta)=\sum_{T\subset\Z\text{ finite}}\int_{\{0,1\}^T}c_T(\eta,d\zeta)[f(\eta^\zeta)-f(\eta)],
\end{equation}
where $c_T(\eta,d\zeta)$ is a finite positive measure on $\{0,1\}^T$, and 
$$
\eta^{\zeta}(x):=\begin{cases}\zeta(x)&\text{if $x \in T$}\\ \eta(x)&\text{if $x \notin T$}\end{cases}
$$
for $\zeta$ $\in \{0,1\}^T$.  The interpretation is that $\eta$ is the current configuration, $c_T(\eta,\{0,1\}^T)$ is the rate at which a transition occurs involving coordinates in $T$, and $c_T(\eta,d\zeta)/c_T(\eta,\{0,1\}^T)$ is the distribution of the restriction to $T$ of the new configuration after a transition.  

We conclude from \eqref{Omega-C'} and \eqref{Omega-general} that
\begin{align*}
c_{\{x\}}(\eta,G)&=\delta_{1-\eta(x)}(G)[\gamma c'(x,\eta)+(1-\gamma) c(x,\eta)],\\
c_{\{x,x+1\}}(\eta,H)&=\delta_{(\eta(x+1),\eta(x))}(H)\frac{\gamma}{2}\bm1_{\{\eta(x)\ne\eta(x+1)\}},
\end{align*}
where $\delta_u$ is the unit mass concentrated at $u$.  Here $G\subset\{0,1\}$ and $H\subset\{(0,0),(0,1),(1,0),(1,1)\}$.

The sufficient condition for ergodicity is that $M<\varepsilon$ (Liggett, 1985, Theorem I.4.1), where $M$ and $\varepsilon$ are constants, defined below, that remain to be evaluated.

For $u\in {\bf Z}$ and finite $T\subset {\bf Z}$, let $c_T(u):=\sup \{\|c_T(\eta,d\zeta)-c_T(\eta',d\zeta)\|_{\text{TV}}: \eta,\eta'\in\Sigma,\eta(y)=\eta'(y)\,\forall\,y\neq u\}$, where $\|\cdot\|_{\text{TV}}$ denotes the total variation norm of a measure on $\{0,1\}^T$.  Then
\begin{align} \label{M-for-C'}
M &:=\sup_{x \in \Z}\sum_{T \ni x}\sum_{u: u \neq x}c_T(u)\nonumber\\ 
&=\sup_{x \in \Z}\sum_{T \ni x}\sum_{u: u \neq x}\sup_{\eta,\eta'\in\Sigma,\eta(y)=\eta'(y)\,\forall\,y\neq u}\|c_T(\eta,d\zeta)-c_T(\eta',d\zeta)\|_{\text{TV}}\nonumber\\ 
&=\sup_{x \in \Z}\bigg[\sum_{u:u\neq x}\sup_{\eta \in \Sigma}\|c_{\{x\}}(\eta,d\zeta)-c_{\{x\}}(\eta_u,d\zeta)\|_{\text{TV}}\nonumber\\
&\qquad\;\;{}+\sum_{v: v\neq x}\sup_{\eta \in \Sigma}\|c_{\{x,x+1\}}(\eta,d\zeta)-c_{\{x,x+1\}}(\eta_v,d\zeta)\|_{\text{TV}}\nonumber\\
&\qquad\;\;{}+\sum_{w: w\neq x}\sup_{\eta \in \Sigma}\|c_{\{x-1,x\}}(\eta,d\zeta)-c_{\{x-1,x\}}(\eta_w,d\zeta)\|_{\text{TV}}\bigg]\nonumber\\ 
&=\sup_{x \in \Z}\bigg[\sum_{u:u\neq x}\sup_{\eta \in \Sigma} \sup_{G \subset \{0,1\}}|c_{\{x\}}(\eta,G)-c_{\{x\}}(\eta_u,G)|\nonumber\\
&\qquad\;\;{}+\sum_{v: v\neq x}\sup_{\eta \in \Sigma}\sup_{H \subset \{(0,0),(0,1),(1,0),(1,1)\}}|c_{\{x,x+1\}}(\eta,H)-c_{\{x,x+1\}}(\eta_v,H)|\nonumber\\
&\qquad\;\;{}+\sum_{w: w\neq x}\sup_{\eta \in \Sigma}\sup_{H \subset \{(0,0),(0,1),(1,0),(1,1)\}}|c_{\{x-1,x\}}(\eta,H)-c_{\{x-1,x\}}(\eta_w,H)|\bigg]\nonumber\\ 
&=\sup_{x \in \Z}\bigg[\sum_{u:u\neq x}\sup_{\eta \in \Sigma} |\gamma c'(x,\eta)+(1-\gamma) c(x,\eta)-[\gamma c'(x,\eta_{u})+(1-\gamma) c(x,\eta_{u})]|\nonumber\\
&\qquad\;\;{}+\sup_{\eta \in \Sigma}\sup_{H \subset \{(0,0),(0,1),(1,0),(1,1)\}}|c_{\{x,x+1\}}(\eta,H)-c_{\{x,x+1\}}(\eta_{x+1},H)|\nonumber\\
&\qquad\;\;{}+\sup_{\eta \in \Sigma}\sup_{H \subset \{(0,0),(0,1),(1,0),(1,1)\}}|c_{\{x-1,x\}}(\eta,H)-c_{\{x-1,x\}}(\eta_{x-1},H)|\bigg]\nonumber\\ 
&=\sup_{x \in \Z}\sum_{u=x\pm1}\sup_{\eta \in \Sigma} |\gamma c'(x,\eta)+(1-\gamma) c(x,\eta)-[\gamma c'(x,\eta_{u})+(1-\gamma) c(x,\eta_{u})]|\nonumber\\
&\qquad\qquad\qquad{}+2\cdot\frac{\gamma}{2}\nonumber\\
&=\sup_{x \in \Z}\bigg[\sup_{\eta \in \Sigma} |\gamma [c'(x,\eta)-c'(x,\eta_{x+1})]+(1-\gamma)[c(x,\eta)- c(x,\eta_{x+1})]|\nonumber\\
&\qquad\qquad{}+\sup_{\eta \in \Sigma} |\gamma [c'(x,\eta)-c'(x,\eta_{x-1})]+(1-\gamma)[c(x,\eta)- c(x,\eta_{x-1})]|\bigg]+\gamma\nonumber\\
&=\max\bigg[ \bigg|\frac{\gamma}{2}+(1-\gamma)(p_0-p_1)\bigg|,\bigg|\frac{\gamma}{2}+(1-\gamma)(p_2-p_3)\bigg|\bigg]\nonumber\\ 
&\qquad{} + \max\bigg[\bigg|\frac{\gamma}{2}+(1-\gamma)(p_0-p_2)\bigg|,\bigg|\frac{\gamma}{2}+(1-\gamma)(p_1-p_3)\bigg|\bigg]+\gamma,
\end{align}
where the last step requires clarification.  Notice first that
$$ 
c'(x,\eta)=\frac{1}{2}[ \bm{1}_{\{\eta(x)=\eta(x+1)\}}+ \bm{1}_{\{\eta(x)=\eta(x-1)\}}],
$$
\begin{align*}
c'(x,\eta_{x+1})&=\frac{1}{2}[ \bm{1}_{\{\eta_{x+1}(x)=\eta_{x+1}(x+1)\}}+ \bm{1}_{\{\eta_{x+1}(x)=\eta_{x+1}(x-1)\}}]\\
&=\frac{1}{2}[ \bm{1}_{\{\eta(x)=1-\eta(x+1)\}}+ \bm{1}_{\{\eta(x)=\eta(x-1)\}}],
\end{align*}
\begin{align*}
 c'(x,\eta_{x-1})&=\frac{1}{2}[ \bm{1}_{\{\eta_{x-1}(x)=\eta_{x-1}(x+1)\}}+ \bm{1}_{\{\eta_{x-1}(x)=\eta_{x-1}(x-1)\}}]\\
 &=\frac{1}{2}[ \bm{1}_{\{\eta(x)=\eta(x+1)\}}+ \bm{1}_{\{\eta(x)=1-\eta(x-1)\}}],
\end{align*}
$$
c(x,\eta)=\begin{cases}p_{m_{x}(\eta)}&\text{if $\eta(x)=0$,}\\ q_{m_{x}(\eta)}&\text{if $\eta(x)=1$,}\end{cases}
$$
$$ 
c(x,\eta_{x+1})=\begin{cases}p_{m_{x}(\eta_{x+1})}&\text{if $\eta(x)=0$,}\\ q_{m_{x}(\eta_{x+1})}&\text{if $\eta(x)=1$,}\end{cases}\quad
c(x,\eta_{x-1})=\begin{cases}p_{m_{x}(\eta_{x-1})}&\text{if $\eta(x)=0$,}\\ q_{m_{x}(\eta_{x-1})}&\text{if $\eta(x)=1$,}\end{cases}
$$
where $q_m:=1-p_m$ for $m=0,1,2,3$ and $m_{x}(\eta):=2\eta(x-1)+\eta(x+1) \in \{0,1,2,3\}$.  The last line of \eqref{M-for-C'} is by direct calculation.\bigskip

Next, 
\begin{align} \label{eps-for-C'}
\varepsilon &:=\inf_{u \in \Z}\;\inf_{\eta=\eta'\text{ off }u,\;\eta(u)\neq\eta'(u)}
\sum_{T \ni u} [c_T(\eta,\{\zeta\in \{0,1\}^{T}: \zeta(u)=\eta'(u)\})\nonumber\\
&\qquad\qquad\qquad\qquad\qquad\qquad\qquad{}+c_T(\eta',\{\zeta\in\{0,1\}^{T}:\zeta(u)=\eta(u)\})]\nonumber\\
&=\inf_{u \in \Z} \inf_{\eta\in \Sigma}\big[c_{\{u\}}(\eta,\{\eta_u(u)\})+c_{\{u\}}(\eta_u,\{\eta(u)\})\nonumber\\
&\qquad\qquad\quad{}+c_{\{u,u+1\}}(\eta,\{\zeta\in\{0,1\}^{\{u,u+1\}}:\zeta(u)=\eta_{u}(u)\})\nonumber\\
&\qquad\qquad\quad{}+c_{\{u,u+1\}}(\eta_{u},\{\zeta\in\{0,1\}^{\{u,u+1\}}:\zeta(u)=\eta(u)\})\nonumber\\
&\qquad\qquad\quad{}+c_{\{u-1,u\}}(\eta,\{\zeta\in\{0,1\}^{\{u-1,u\}}:\zeta(u)=\eta_{u}(u)\})\nonumber\\
&\qquad\qquad\quad{}+c_{\{u-1,u\}}(\eta_{u},\{\zeta\in\{0,1\}^{\{u-1,u\}}:\zeta(u)=\eta(u)\})\big]\nonumber\\
&=\inf_{u \in \Z} \inf_{\eta \in \Sigma}\big[c_{\{u\}}(\eta, \{1-\eta (u)\})+c_{\{u\}}(\eta_{u},\{1-\eta_{u}(u)\})\nonumber\\
&\qquad\qquad\quad{}+c_{\{u,u+1\}}(\eta,\{\zeta\in\{0,1\}^{\{u,u+1\}}:\zeta(u)=1-\eta(u)\})\nonumber\\
&\qquad\qquad\quad{}+c_{\{u,u+1\}}(\eta_{u},\{\zeta\in\{0,1\}^{\{u,u+1\}}:\zeta(u)=1-\eta_u(u)\})\nonumber\\
&\qquad\qquad\quad{}+c_{\{u-1,u\}}(\eta,\{\zeta\in\{0,1\}^{\{u-1,u\}}:\zeta(u)=1-\eta(u)\})\nonumber\\
&\qquad\qquad\quad{}+c_{\{u-1,u\}}(\eta_{u},\{\zeta\in\{0,1\}^{\{u-1,u\}}:\zeta(u)=1-\eta_u(u)\})\big]\nonumber\\
&=\inf_{u \in \Z} \inf_{\eta \in \Sigma}\big[\gamma[c'(u,\eta)+c'(u,\eta_{u})]+(1-\gamma)[c(u,\eta)+c(u,\eta_{u})]\big]+2\cdot\frac{\gamma}{2}\nonumber\\
&=1+\gamma,
\end{align} 
where the last line of \eqref{eps-for-C'} follows from
\begin{align*}
c'(u,\eta)+c'(u,\eta_{u})&=\frac{1}{2}[ \bm{1}_{\{\eta(u)=\eta(u+1)\}}+ \bm{1}_{\{\eta(u)=\eta(u-1)\}}]\\
&\quad{}+\frac{1}{2}[ \bm{1}_{\{\eta_{u}(u)=\eta_{u}(u+1)\}}+ \bm{1}_{\{\eta_{u}(u)=\eta_{u}(u-1)\}}]\\
&=\frac{1}{2}[ \bm{1}_{\{\eta(u)=\eta(u+1)\}}+ \bm{1}_{\{\eta(u)=\eta(u-1)\}}]\\
&\quad{}+\frac{1}{2}[ \bm{1}_{\{\eta(u)=1-\eta(u+1)\}}+ \bm{1}_{\{\eta(u)=1-\eta(u-1)\}}]\\
&=1
\end{align*}
and $c(u,\eta)+c(u,\eta_{u})=1$.  We have proved the following theorem.

\begin{theorem}
The interacting particle system in $\Sigma:=\{0,1\}^\Z$ with generator $\Omega_{C'}:=\gamma\Omega_{A'}+(1-\gamma)\Omega_B$, where $0<\gamma<1$, is ergodic if
\begin{align}\label{erg-ineq}
&\max\bigg[ \bigg|\frac{\gamma}{2}+(1-\gamma)(p_0-p_1)\bigg|,\bigg|\frac{\gamma}{2}+(1-\gamma)(p_2-p_3)\bigg|\bigg]\nonumber\\ 
&\qquad + \max\bigg[\bigg|\frac{\gamma}{2}+(1-\gamma)(p_0-p_2)\bigg|,\bigg|\frac{\gamma}{2}+(1-\gamma)(p_1-p_3)\bigg|\bigg]<1.
\end{align}
\end{theorem}

The volume of the subset of the parameter space $[0,1]^4$ for which \eqref{erg-ineq} holds with $\gamma=1/2$ is 5/6.  If we assume that $p_1=p_2$, then the volume of the subset of the parameter space $[0,1]^3$ for which \eqref{erg-ineq} holds with $\gamma=1/2$ is 3/4.  In fact, the volume is $3/4$ if and only if $\gamma \ge 1/3$.

\section{Convergence of means}\label{sec-convergence}

We would like to prove that $\lim_{N\rightarrow \infty}\mu^N_{(\gamma,1-\gamma)'}$ and $\lim_{N\rightarrow \infty}\mu^N_{[r,s]'}$ exist under certain conditions, where $\mu^N_{(\gamma,1-\gamma)'}$ denotes the mean profit per turn at equilibrium to the $N$ players playing the $(\gamma,1-\gamma)$ random mixture of games $A'$ and $B$ (the Parrondo games of Xie et al., 2011), and $\mu^N_{[r,s]'}$ denotes the mean profit per turn at equilibrium to the $N$ players playing games $A'$ and $B$ in the nonrandom periodic pattern $A', A', \dots, A'$ ($r$ times), $B, B, \dots, B$ ($s$ times), $A', A', \dots, A'$ ($r$ times), $B, B, \dots, B$ ($s$ times), and so on.  The first result is relatively straightforward, while the second requires more work.  The key step for the second result is to prove that the sequence of discrete generators converges to the generator of an interacting particle system.  

We want to show that our sequence of discrete-time Markov chains, suitably rescaled, converges in distribution to an interacting particle system on $\Z$. The limiting process is characterized in terms of its generator. First, we need to define generators corresponding to  game $A'$, game $B$, and game $C'$.  The state space is 
\begin{align*}
\Sigma&:=\{0,1\}^{\Z}\\
&\phantom{:}=\{\eta=(\ldots,\eta(-2),\eta(-1),\eta(0),\eta(1),\eta(2),\ldots):\eta(x)\in\{0,1\}\text{ for all }x\in\Z\}.
\end{align*}
In Section~\ref{sec-intro} we defined $\eta^{x,x\pm1,\pm1}$ for $\eta\in\{0,1\}^N$, and that definition is easily extended to $\eta\in\Sigma$.  For $\eta\in\Sigma$ and $x\in\Z$ define $\eta^{x,-1}$ and  $\eta^{x,1}$ to be the elements of $\Sigma$ given by
$$
\eta^{x,-1}(y):=\begin{cases}1&\text{if $y=x-1$,}\\ 0&\text{if $y=x$,}\\ \eta(y)&\text{otherwise,}\end{cases}\qquad
\eta^{x,1}(y):=\begin{cases}0&\text{if $y=x$,}\\ 1&\text{if $y=x+1$,}\\ \eta(y)&\text{otherwise.}\end{cases}
$$
For example, $\eta^{0,1}:=(\ldots,\eta(-2),\eta(-1), 0, 1, \eta(2),\eta(3), \ldots)$. And let $\eta_x$ be the element of $\Sigma$ equal to $\eta$ except at the $x$th coordinate.  Then the generators are 
\begin{align*}
&(\Omega_{A'}f)(\eta)\\
&:=\sum_{x\in\Z}\bigg[\frac14 f(\eta^{x,x-1,-1})+\frac14 f(\eta^{x,x-1,1})+\frac14 f(\eta^{x,x+1,-1})+\frac14 f(\eta^{x,x+1,1})-f(\eta)\bigg]\\
&=\frac12\sum_{x\in\Z}[f(\eta^{x-1,x,1})-f(\eta)]+\frac12\sum_{x\in\Z}[f(\eta^{x,x+1,-1})-f(\eta)]\\
&=\frac12\bigg[\sum_{x:(\eta(x-1),\eta(x))=(0,0)}[f(\eta_{x-1})-f(\eta)]+\sum_{x:(\eta(x-1),\eta(x))=(0,1)}[f(_{x-1}\eta_x)-f(\eta)]\\
&\quad{}+\sum_{x:(\eta(x-1),\eta(x))=(1,1)}[f(\eta_x)-f(\eta)]+\sum_{x:(\eta(x),\eta(x+1))=(0,0)}[f(\eta_{x+1})-f(\eta)]\\
&\quad{}+\sum_{x:(\eta(x),\eta(x+1))=(1,0)}[f(_x\eta_{x+1})-f(\eta)]+\sum_{x:(\eta(x),\eta(x+1))=(1,1)}[f(\eta_x)-f(\eta)]\bigg]\\
&=\frac12\bigg[\sum_{x:(\eta(x),\eta(x+1))=(0,0)}[f(\eta_x)-f(\eta)]+\sum_{x:(\eta(x),\eta(x+1))=(0,1)}[f(_x\eta_{x+1})-f(\eta)]\\
&\quad{}+\sum_{x:(\eta(x),\eta(x+1))=(1,1)}[f(\eta_{x+1})-f(\eta)]+\sum_{x:(\eta(x),\eta(x+1))=(0,0)}[f(\eta_{x+1})-f(\eta)]\\
&\quad{}+\sum_{x:(\eta(x),\eta(x+1))=(1,0)}[f(_x\eta_{x+1})-f(\eta)]+\sum_{x:(\eta(x),\eta(x+1))=(1,1)}[f(\eta_x)-f(\eta)]\bigg]\\
&=\frac12\bigg[\sum_{x:\eta(x)=\eta(x+1)}[f(\eta_x)-f(\eta)]+\sum_{x:\eta(x)=\eta(x+1)}[f(\eta_{x+1})-f(\eta)]\\
&\qquad\quad{}+\sum_{x\in\Z}[f(_x\eta_{x+1})-f(\eta)]\bigg]\\
&=\frac12\bigg[\sum_{x:\eta(x)=\eta(x+1)}[f(\eta_x)-f(\eta)]+\sum_{x:\eta(x-1)=\eta(x)}[f(\eta_x)-f(\eta)]\\
&\qquad\quad{}+\sum_{x\in\Z}[f(_x\eta_{x+1})-f(\eta)]\bigg]\\
&=\sum_{x\in\Z} c'(x,\eta)[f(\eta_x)-f(\eta)]+\frac12\sum_{x\in\Z}[f(_x\eta_{x+1})-f(\eta)],
\end{align*}
where $c'(x,\eta)$ is as in \eqref{c'-def}, 
\begin{equation*}
(\Omega_{B}f)(\eta):=\sum_{x \in \Z}c(x,\eta)\big[f(\eta_x)-f(\eta)\big],
\end{equation*}
where $c(x,\eta)$ is as in \eqref{c-def}, and
\begin{equation*}
(\Omega_{C'}f)(\eta):=\big[\gamma\Omega_{A'}f+(1-\gamma)\Omega_{B}f\big](\eta)
\end{equation*}
for functions $f\in C(\Sigma)$ depending on only finitely many coordinates.

Next, it is necessary to show that this interacting particle system is the limit in distribution of the $N$-player model as $N\to\infty$.  Furthermore, we need to adjust the state space by relabeling the players. Specifically, we let 
$$
\Sigma_N :=\{\eta =(\eta(l_{N}), \ldots, \eta(r_N)):\eta(x) \in \{0, 1\}\text{ for }x=l_N,\ldots, r_N \},
$$
where $l_N$ and $r_N$ are as in \eqref{relabel}.  It should be noted that players $l_N$ and $r_N$ are nearest neighbors.  We denote the Markov chain in $\Sigma_N$ by $\{X_k^{\cdot,N},\; k=0,1,2,\ldots\}$, where $A'$, $B$, or $C'$ appears in place of the dot.

First, let us analyze game $A'$.
The one-step transition matrix $\bm P_{A'}$ of the Markov chain in the state space $\Sigma_N$ has the form
\begin{align*}
P_{A'}(\eta, \xi)&:=\frac{1}{4N}\sum_{l_N\le x\le r_N}[\delta(\eta^{x,x-1,-1},\xi)+\delta(\eta^{x,x-1,1},\xi)\\
&\qquad\qquad\qquad\qquad{}+\delta(\eta^{x,x+1,-1},\xi)+\delta(\eta^{x,x+1,1},\xi)]\\
&\;=\frac{1}{2N}\sum_{l_N\le x\le r_N}\big[\delta(\eta^{x,-1},\xi)+\delta(\eta^{x,1},\xi)\big] ,
\end{align*}
where $\delta(\eta, \xi)$ is the Kronecker delta, which is $1$ if $\eta = \xi$ and is $0$ otherwise; the sum over $x$  ranges over $\{ l_{N}, \dots, r_{N} \}$, and $l_N-1:=r_N$ and $r_N+1:=l_N$.  Next,  we have the one-step transition matrix $\bm P_B$ of the form 
$$
P_{B}(\xi, \zeta):=\frac{1}{N} \sum_y [1-c(y,\xi)]\delta(\xi, \zeta)+\frac{1}{N} \sum_y c(y,\xi)\delta(\xi_{y}, \zeta),
$$
where  the sum over  $y$ also ranges over $\{ l_{N}, \dots, r_{N} \}$; $c(y,\xi)$ is as in \eqref{c-def}, except that $l_N-1:=r_N$ and $r_N+1:=l_N$.  

We speed up time in the $N$-player model so that $N$ one-step transitions occur per unit of time. Then the discrete generator corresponding to game $A'$ is 
\begin{align*}
(\Omega^N_{A'}f)(\eta)&=N\text{E}\big[f(X_1^{A',N})-f(\eta)\mid X_0^{A',N}=\eta\big]\\
&=N\sum_{\xi\in\Sigma_N}P_{A'}(\eta, \xi)\big[f(\xi)-f(\eta)\big]\\
&=N\sum_{\xi \in \Sigma_N}\frac{1}{2N}\sum_{l_N\leq x \leq r_N}\big[\delta(\eta^{x,-1}, \xi)+\delta(\eta^{x,1}, \xi) \big]\big[f(\xi)-f(\eta)\big]\\
&=\frac12 \sum_{l_N\leq x \leq r_N}\big[f(\eta^{x,-1})-f(\eta)+f(\eta^{x,1})-f(\eta)\big]\\
&=\sum_{l_N\leq x \leq r_N}\bigg[\frac12f(\eta^{x,-1})+\frac12f(\eta^{x,1})-f(\eta)\bigg].
\end{align*}
The discrete generator corresponding to game $B$ is 
\begin{align*}
(\Omega^N_{B}f)(\eta)&=N\text{E}\big[f(X_1^{B,N})-f(\eta)\mid X_0^{B,N}=\eta\big]\\
&=N\sum_{\xi\in\Sigma_N}P_B(\eta, \xi)\big[f(\xi)-f(\eta)\big]\\
&=\sum_{l_N\leq x \leq r_N: \eta(x)=0}p_{m_x(\eta)}\big[f(\eta_x)-f(\eta)\big]\\
&\qquad\quad+\sum_{l_N\leq x \leq r_N: \eta(x)=1}q_{m_x(\eta)}\big[f(\eta_x)-f(\eta)\big].
\end{align*}
Hence the discrete generator corresponding to game $C'$ is 
\begin{align*}
(\Omega^N_{C'}f)(\eta)&=N\text{E}\big[f(X_1^{C',N})-f(\eta)\mid X_0^{C',N}=\eta\big]\\
&=\big[\gamma\Omega_{A'}^N f+(1-\gamma)\Omega_B^N f\big](\eta)\\
&=\gamma\sum_{l_N\leq x \leq r_N}\bigg[\frac12f(\eta^{x,-1})+\frac12f(\eta^{x,1})-f(\eta)\bigg]\\
&\qquad+(1-\gamma)\bigg[\sum_{l_N\leq x \leq r_N: \eta(x)=0}p_{m_x(\eta)}\big[f(\eta_x)-f(\eta)\big]\\
&\qquad\qquad\qquad\qquad+\sum_{l_N\leq x \leq r_N: \eta(x)=1}q_{m_x(\eta)}\big[f(\eta_x)-f(\eta)\big]\bigg].
\end{align*}

We define $\psi_N: B(\Sigma)\mapsto B(\Sigma_N)$ by 
\begin{equation*}
(\psi_N f)(\eta(l_N),\ldots,\eta(r_N)):=f(\ldots,1,1,\eta(l_N),\ldots,\eta(r_N),1,1,\ldots).
\end{equation*}

\begin{lemma} \label{Omega-psi}
If $f\in C(\Sigma)$ depends on $\eta$ only through the $2K+1$ components $\eta(x)$ for $-K \leq x \leq K$, then
\begin{equation*}
(\Omega^N_{A'}\psi_Nf)(\eta)=\psi_N(\Omega_{A'}f)(\eta),
\end{equation*}
\begin{equation*}
(\Omega^N_{B}\psi_Nf)(\eta)=\psi_N(\Omega_{B}f)(\eta),
\end{equation*}
and
\begin{equation*}
(\Omega^N_{C'}\psi_Nf)(\eta)=\psi_N(\Omega_{C'}f)(\eta) 
\end{equation*}
for all $\eta \in \Sigma_N$ and $N\geq 2K+4$.
\end{lemma} 

\begin{proof}
The proof is straightforward.
\end{proof}

Lemma~\ref{Omega-psi} implies that the process $\{X_{\lfloor N t \rfloor}^{C',N}\}$ converges in distribution to the interacting particle system $\{X_t\}$ by Theorem 1.6.5 and 4.2.6 of Ethier and Kurtz (1986).  More importantly, it implies that, if the interacting particle system has a unique stationary distribution, then the unique stationary distribution of the $N$-player Markov chain converges to it in the topology of weak convergence, essentially by Proposition I.2.14 of Liggett (1985).  Let us assume that the interacting particle system with generator $\Omega_{C'}$ has a unique stationary distribution $\pi$, and let us denote the unique stationary distribution of the $N$-player Markov chain for the $(\gamma,1-\gamma)$ random mixture of games $A'$ and $B$ by $\pi^N$.  Let us denote their $-1,1$ two-dimensional marginals by $\pi_{-1,1}^N$ and $\pi_{-1,1}$.  Then we have
\begin{align}\label{mean-limit}
\mu_{(\gamma,1-\gamma)'}^N&=(1-\gamma)[\pi^N_{-1,1}(0,0)(2p_0-1)+\pi^N_{-1,1}(0,1)(2p_1-1)\nonumber\\
&\qquad{}+\pi^N_{-1,1}(1,0)(2p_2-1)+\pi^N_{-1,1}(1,1)(2p_3-1)]\nonumber\\
&\to(1-\gamma)[\pi_{-1,1}(0,0)(2p_0-1)+\pi_{-1,1}(0,1)(2p_1-1)\nonumber\\
&\qquad{}+\pi_{-1,1}(1,0)(2p_2-1)+\pi_{-1,1}(1,1)(2p_3-1)]\nonumber\\
&=:\mu_{(\gamma,1-\gamma)'},
\end{align}
where the first equality is based on the following idea:  Suppose player 0 is chosen to play. If he plays game $A'$ (probability $\gamma$) he wins or loses 1 with probability 1/2 each, so his mean profit is 0; if he plays game $B$ (probability $1-\gamma$), he tosses a $p_m$ coins with $m$ determined by the status of his nearest neighbors, hence his expected profit is $2p_m-1$.  We conclude that $\mu_{(\gamma,1-\gamma)'}^N$, the mean profit per turn at equilibrium to the $N$ players playing the $(\gamma,1-\gamma)$ random mixture of games $A'$ and $B$, converges as $N\to\infty$ to a limit that can be expressed in terms of an interacting particle system. We have proved the following.

\begin{theorem}\label{conv-N-gamma 1-gamma}
Fix $\gamma\in(0,1)$.  Assume that the interacting particle system on ${\bf Z}$ with generator $\Omega_{C'}:=\gamma\Omega_{A'}+(1-\gamma)\Omega_B$ is ergodic with unique stationary distribution $\pi$.  Then $\lim_{N\to\infty}\mu_{(\gamma,1-\gamma)'}^N=\mu_{(\gamma,1-\gamma)'}$, where $\mu_{(\gamma,1-\gamma)'}$ is as in \eqref{mean-limit}.
\end{theorem}

The discrete generator for the nonrandom periodic pattern $(A')^r B^s$ has the form, for $f\in B(\Sigma_N)$,
$$
(\Omega^N_{[r, s]'}f)(\eta^0)=\frac{N}{r+s}\sum_{\eta^{r+s}}[f(\eta^{r+s})-f(\eta^0)](\bm{P}^r_{A'}\bm{P}^s_{B})(\eta^0, \eta^{r+s}).
$$
We begin by evaluating
\begin{align}\label{P_A^rP_B^s}
&\!\!\!\!\!(\bm{P}^{r}_{A'}\bm{P}^{s}_{B})(\eta^0, \eta^{r+s})\nonumber\\
&=\sum_{\eta^1,\eta^2,\ldots, \eta^{r+s-1} } P_{A'}(\eta^0, \eta^1)P_{A'}(\eta^1, \eta^2)\cdots P_{A'}(\eta^{r-1}, \eta^r)\nonumber\\
&\quad\cdot P_B(\eta^r, \eta^{r+1})P_B(\eta^{r+1}, \eta^{r+2})\cdots P_B(\eta^{r+s-1}, \eta^{r+s})\nonumber\\
&=\sum_{\eta^1,\eta^2,\ldots, \eta^{r+s-1} }\prod_{i=1}^r P_{A'}(\eta^{i-1},\eta^i)\prod_{i=r+1}^{r+s}P_B(\eta^{i-1},\eta^i)\nonumber\\
&=\sum_{\eta^1,\eta^2,\ldots,\eta^{r+s-1}} \prod_{i=1}^{r}\bigg[\frac{1}{2N}\sum_{x_{i}}[\delta((\eta^{i-1})^{x_{i},-1}, \eta^i)+\delta((\eta^{i-1})^{x_{i},1}, \eta^i)]\bigg]\nonumber\\
&\quad\cdot \prod_{i=r+1}^{r+s}\bigg[\frac{1}{N}\sum_{x_{i}}(1-c(x_{i},\eta^{i-1}))\delta(\eta^{i-1}, \eta^{i})+\frac{1}{N}\sum_{x_{i}}c(x_{i},\eta^{i-1})\delta((\eta^{i-1})_{x_{i}}, \eta^{i})\bigg]\nonumber\\
&=\frac{1}{2^rN^{r+s}}\sum_{A\subset \{1,\ldots, r\}}\sum_{B\subset \{r+1,\ldots, r+s\}}\sum_{\eta^1,\eta^2,\ldots, \eta^{r+s-1} }
\prod_{i\in A^c}\bigg[\sum_{x_{i}}\delta((\eta^{i-1})^{x_{i},-1}, \eta^i)\bigg]\nonumber\\  
&\quad\cdot \prod_{i\in A}\bigg[\sum_{x_{i}}\delta((\eta^{i-1})^{x_{i},1}, \eta^i)\bigg] \prod_{i\in B^c}\bigg[\sum_{x_{i}}(1-c(x_{i},\eta^{i-1}))\delta(\eta^{i-1}, \eta^{i})\bigg]\nonumber\\
&\quad\cdot \prod_{i\in B}\bigg[\sum_{x_{i}}c(x_{i},\eta^{i-1})\delta((\eta^{i-1})_{x_{i}}, \eta^{i})\bigg]\nonumber\\
&=\frac{1}{2^rN^{r+s}}\sum_{A\subset \{1,\ldots, r\}}\sum_{B\subset \{r+1,\ldots, r+s\}}\sum_{\eta^1,\eta^2,\ldots, \eta^{r+s-1}}\sum_{x_{i}: i\in A^c}\sum_{x_{i}: i\in A }\nonumber\\
&\quad\cdot \prod_{j\in A^c}\big[\delta((\eta^{j-1})^{x_{j},-1}, \eta^j)\big] \prod_{j\in A}\big[\delta((\eta^{j-1})^{x_{j},1}, \eta^j)\big]\sum_{x_{i}: i\in B^c}\sum_{x_{i}: i\in B}\nonumber\\
&\quad\cdot \prod_{j\in B^c}\big[(1-c(x_{j},\eta^{j-1}))\delta(\eta^{j-1}, \eta^{j})\big] \prod_{j\in B}\big[c(x_{j},\eta^{j-1})\delta((\eta^{j-1})_{x_{j}}, \eta^{j})\big]\nonumber\\
&=\frac{1}{2^rN^{r+s}}\sum_{A\subset \{1,\ldots, r\}}\sum_{B\subset \{r+1,\ldots, r+s\}}\sum_{x_{i}: i\in \{1,2,\ldots, r+s\}}\nonumber\\
&\quad\cdot\prod_{j\in B^c}\big[1-c(x_{j},((\cdots((\cdots(((\eta^{0})^{x_{1},a_1})^{x_{2},a_2})\cdots)^{x_{p},a_p})\cdots)^{x_{r},a_r})_{\{x_{l}: l \in B,l<j\}})\big]\nonumber\\
&\quad\cdot\prod_{j\in B}c(x_{j},((\cdots((\cdots(((\eta^{0})^{x_{1},a_1})^{x_{2},a_2})\cdots)^{x_{p},a_p})\cdots)^{x_{r},a_r})_{\{x_{l}: l \in B,l<j\}})\nonumber\\
&\qquad\times\delta(((\cdots((\cdots(((\eta^{0})^{x_{1},a_1})^{x_{2},a_2})\cdots)^{x_{p},a_p})\cdots)^{x_{r},a_r})_{\{x_{l}: l \in B\}}, \eta^{r+s}),
\end{align}
where $A^c:=\{1,2,\ldots, r\}-A$ and $B^c:=\{r+1,r+2,\ldots, r+s\}-B$; also $p \in \{1,2,\ldots, r\}$ and 
$$
a_p = \begin{cases} -1 &\text{if $p \in A^c$}, \\ \phantom{-}1 & \text{if $p \in A$}. \end{cases} 
$$
Here, for example, $\eta_{\{x_l:l\in B\}}$ denotes $\eta$ with the spins flipped at each site $x_l$ with $l\in B$.  These site labels need not be distinct, so if there are multiple flips at a single site, only the parity of the number of flips is relevant.

Next, assume that $f\in B(\Sigma)$ depends only on $\eta(-(K-2)),\ldots,\eta(K-2)$ for some integer $K\ge2$, and put $f_N:=\psi_N f\in B(\Sigma_N)$.  Then the discrete generator for the pattern $(A')^r B^s$, acting on $f_N$, reduces to   
\begin{align}\label{Omega_[r,s]}
&(\Omega^N_{[r, s]'}f_N)(\eta^0)\nonumber\\
&=\frac{N}{r+s}\sum_{\eta^{r+s}}[f_N(\eta^{r+s})-f_N(\eta^0)](\bm{P}^{r}_{A'}\bm{P}^{s}_{B})(\eta^0, \eta^{r+s})\nonumber\\
&=\frac{1}{r+s}\ \frac{1}{2^rN^{r+s-1}}\sum_{A\subset \{1,\ldots, r\}}\sum_{B\subset \{r+1,\ldots, r+s\}}\sum_{x_{i}: i\in \{1,2,\ldots, r+s\}}\nonumber\\
&\quad\cdot \prod_{j\in B^c}\big[1-c(x_{j},((\cdots((\cdots(((\eta^{0})^{x_{1},a_1})^{x_{2},a_2})\cdots)^{x_{p},a_p})\cdots)^{x_{r},a_r})_{\{x_{l}: l \in B,l<j\}})\big]\nonumber\\
&\quad\cdot \prod_{j\in B}c(x_{j},((\cdots((\cdots(((\eta^{0})^{x_{1},a_1})^{x_{2},a_2})\cdots)^{x_{p},a_p})\cdots)^{x_{r},a_r})_{\{x_{l}: l \in B,l<j\}})\nonumber\\
&\quad\cdot [f_N(((\cdots((\cdots(((\eta^{0})^{x_{1},a_1})^{x_{2},a_2})\cdots)^{x_{p},a_p})\cdots)^{x_{r},a_r})_{\{x_{l}: l \in B\}})-f_N( \eta^{0})]
\end{align}
We replace $\sum_{x_{i}: i\in \{1,2,\ldots, r+s\}}$ in \eqref{Omega_[r,s]} by 
\begin{align}\label{decomposition}
&\bigg(\sum_{i\in A^c}+\sum_{i\in A}\bigg)\sum_{|x_i|\leq K}\sum_{|x_m|>K: m\in \{1,2,\ldots,r+s\},m \neq i}\nonumber\\
&\quad{} + \sum_{i\in B}\sum_{|x_i|\leq K}\sum_{|x_m|>K: m\in \{1,2,\ldots,r+s\},m \neq i}
\end{align}
since each sum $\sum_{x_{i}}$ can be written as $\sum_{|x_i|\leq K}+\sum_{|x_i|> K}$ resulting in $2^{r+s}$ multiple sums in which each of those multiple sums with two or more sums of the form $\sum_{|x_i|\leq K}$ contributes at most $O(N^{-1})$  and those without the form $\sum_{|x_i|\leq K},$ where $i \in \{1,2,\ldots,r\}\cup B$ are $0$.  So it is enough to analyze the cases in which only one of the $|x_i|$'s is less than or equal to $K$.

We consider first the first term in \eqref{decomposition}.  It contributes 
\begin{align*}
&\frac{1}{r+s}\;\frac{1}{2^rN^{r+s-1}}\sum_{A\subset \{1,\ldots, r\}}\sum_{B\subset \{r+1,\ldots, r+s\}}\bigg[\sum_{i \in A^c}\sum_{|x_{i}|\leq K}[f_N((\eta^0)^{x_i,-1})- f_N(\eta^0)]\\
&\qquad\qquad\qquad\qquad\qquad\qquad\qquad\qquad\qquad+\sum_{i \in A}\sum_{|x_{i}|\leq K}[f_N((\eta^0)^{x_i,1})-f_N(\eta^0)]\bigg]\\
&\quad\cdot\sum_{|x_m|>K: m\in \{1,2,\ldots,r+s\},m \neq i}\prod_{j\in B^c}\big[1\\
&\qquad\qquad\quad{}-c(x_{j},((\cdots((\cdots(((\eta^{0})^{x_{1},a_1})^{x_{2},a_2})\cdots)^{x_{p},a_p})\cdots)^{x_{r},a_r})_{\{x_{l}: l \in B,l<j\}})\big]\\
&\quad\cdot\prod_{j\in B}c(x_{j},((\cdots((\cdots(((\eta^{0})^{x_{1},a_1})^{x_{2},a_2})\cdots)^{x_{p},a_p})\cdots)^{x_{r},a_r})_{\{x_{l}: l \in B,l<j\}})\\
&=\frac{1}{r+s}\,\frac{1}{2^rN^{r+s-1}}\!\!\!\sum_{A\subset\{1,\ldots,r\}}\sum_{B\subset\{r+1,\ldots,r+s\}}\bigg[\sum_{i\in A^c}\sum_{|x_{i}|\leq K}[f_N((\eta^0)^{x_i,-1})-f_N(\eta^0)]\\
&\qquad\qquad\qquad\qquad\qquad\qquad\qquad\qquad\qquad\;\;{}+\sum_{i \in A}\sum_{|x_{i}|\leq K}[f_N((\eta^0)^{x_i,1})-f_N(\eta^0)]\bigg]\\
&\quad\cdot\sum_{|x_m|>K: m\in \{1,2,\ldots,r+s\},m \neq i}\prod_{j\in B^c}\big[1\\
&\qquad\qquad\qquad\qquad\qquad{}-c(x_{j},(\cdots((\cdots(((\eta^{0})^{x_{1},a_1})^{x_{2},a_2})\cdots)^{x_{p},a_p})\cdots)^{x_{r},a_r})\big]\\
&\quad\cdot\prod_{j\in B}c(x_{j},(\cdots((\cdots(((\eta^{0})^{x_{1},a_1})^{x_{2},a_2})\cdots)^{x_{p},a_p})\cdots)^{x_{r},a_r})\\
&=\frac{1}{r+s}\,\frac{1}{2^rN^{r+s-1}}\big[N-(2K+1)\big]^{r+s-1}\\
&\qquad\qquad\cdot{}\sum_{A\subset \{1,\ldots, r\}}\bigg[|A^c|\sum_{|x|\leq K}[f_N((\eta^0)^{x,-1})- f_N(\eta^0)]\\
&\qquad\qquad\qquad\qquad\qquad{}+|A|\sum_{|x|\leq K}[f_N((\eta^0)^{x,1})-f_N(\eta^0)]\bigg]\\
&=\frac{1}{r+s}\ \frac{1}{2^rN^{r+s-1}}\big[N-(2K+1)\big]^{r+s-1}\sum_{A\subset \{1,\ldots, r\}}\bigg[|A^c|\sum_{|x|\leq K}f_N((\eta^0)^{x,-1})\\
&\qquad\qquad\qquad\qquad\qquad\qquad\qquad\qquad{}+|A|\sum_{|x|\leq K}f_N((\eta^0)^{x,1})-r\sum_{|x|\leq K}f_N(\eta^0)\bigg]\\
&=\frac{1}{r+s}\ \frac{1}{N^{r+s-1}}\big[N-(2K+1)\big]^{r+s-1} \bigg[\frac{r}{2}\sum_{|x|\leq K} f_N((\eta^0)^{x,-1})\\
&\quad\qquad\qquad\qquad\qquad\qquad\qquad\qquad\qquad+\frac{r}{2}\sum_{|x|\leq K} f_N((\eta^0)^{x,1})-r\sum_{|x|\leq K} f_N(\eta^0)\bigg]\\
&=\frac{r}{r+s}\sum_{|x|\leq K}\bigg[\frac{1}{2} f_N((\eta^0)^{x,-1})+\frac{1}{2} f_N((\eta^0)^{x,1})-f_N(\eta^0)\bigg]+O(N^{-1})\\
&=\frac{r}{r+s}\psi_N (\Omega_{A'}f)(\eta^0)+O(N^{-1}),
\end{align*}
where, in the second equality,
\begin{align*}
&\sum_{|x_m|>K: m\in \{1,2,\ldots,r+s\},m \neq i}\sum_{B\subset \{r+1,\ldots, r+s\}}\prod_{j\in B^c}\big[1\\
&\qquad\qquad\qquad\qquad\qquad{}-c(x_{j},(\cdots((\cdots(((\eta^{0})^{x_{1},a_1})^{x_{2},a_2})\cdots)^{x_{p},a_p})\cdots)^{x_{r},a_r})\big]\\
&\qquad\qquad\qquad\qquad\cdot \prod_{j\in B}c(x_{j},(\cdots((\cdots(((\eta^{0})^{x_{1},a_1})^{x_{2},a_2})\cdots)^{x_{p},a_p})\cdots)^{x_{r},a_r})\\
&=\sum_{|x_m|>K: m\in \{1,2,\ldots,r+s\},m \neq i} \prod_{j=r+1}^{r+s}\big[1\nonumber\\
&\qquad\qquad\qquad\qquad{}-c(x_{j},(\cdots((\cdots(((\eta^{0})^{x_{1},a_1})^{x_{2},a_2})\cdots)^{x_{p},a_p})\cdots)^{x_{r},a_r})\\
&\qquad\qquad\qquad\qquad{}+c(x_{j},(\cdots((\cdots(((\eta^{0})^{x_{1},a_1})^{x_{2},a_2})\cdots)^{x_{p},a_p})\cdots)^{x_{r},a_r})\big]\\
&=[N-(2K+1)]^{r+s-1},
\end{align*}
and in the fourth equality, 
\begin{align*}
&\frac{1}{2^r}\sum_{A\subset \{1,\ldots, r\}}|A^c|=\frac{1}{2^r}\sum_{A\subset \{1,\ldots, r\}}|A|=\frac{1}{2^r}\sum_{i=0}^{r}i\ {r\choose i}=\frac{r}{2}.
\end{align*}

Next, we consider the second term in \eqref{decomposition}.  It contributes 
\begin{align*}
&\!\!\!\!\!\frac{1}{r+s}\ \frac{1}{2^rN^{r+s-1}}\sum_{A\subset \{1,\ldots, r\}}\sum_{B\subset \{r+1,\ldots, r+s\}}\sum_{i \in B}\sum_{|x_{i}|\leq K}[f_N((\eta^0)_{x_i})- f_N(\eta^0)]\\
&\quad\cdot\sum_{|x_m|>K: m\in \{1,2,\ldots,r+s\},m \neq i}\prod_{j\in B^c}\big[1\\
&\qquad\qquad{}-c(x_{j},((\cdots((\cdots(((\eta^{0})^{x_{1},a_1})^{x_{2},a_2})\cdots)^{x_{p},a_p})\cdots)^{x_{r},a_r})_{\{x_{l}: l \in B,l<j\}})\big]\\
&\quad\cdot \prod_{j\in B}c(x_{j},((\cdots((\cdots(((\eta^{0})^{x_{1},a_1})^{x_{2},a_2})\cdots)^{x_{p},a_p})\cdots)^{x_{r},a_r})_{\{x_{l}: l \in B,l<j\}})\\
&=\frac{1}{r+s}\ \frac{1}{2^rN^{r+s-1}}\sum_{A\subset\{1,\ldots, r\}}\sum_{i=r+1}^{r+s}\ \sum_{|x_{i}|\leq K}c(x_i,\eta^0)[f_N((\eta^0)_{x_i})- f_N(\eta^0)]\\
&\quad\cdot\sum_{|x_m|>K: m\in \{1,2,\ldots,r+s\},m \neq i}\sum_{B\subset \{r+1,\ldots, r+s\}: i\in B}\prod_{j\in B^c}\big[1-c(x_{j},\eta^{0})\big]\\
&\qquad\qquad\qquad\qquad\qquad\qquad\qquad\qquad\qquad\qquad{}\cdot\prod_{j\in B-\{i\}}c(x_{j},\eta^{0})+O(N^{-1})\\
&=\frac{1}{r+s}\ \frac{1}{2^rN^{r+s-1}}\big[N-(2K+1)\big]^{r+s-1}\sum_{A\subset \{1,\ldots, r\}} \sum_{i=r+1}^{r+s}\\
&\qquad\qquad\qquad\qquad\qquad\qquad{}\cdot\sum_{|x_{i}|\leq K} c(x_i,\eta^0)[f_N((\eta^0)_{x_i})- f_N(\eta^0)]+O(N^{-1})\\
&=\frac{1}{r+s}\ \frac{1}{2^rN^{r+s-1}}\big[N-(2K+1)\big]^{r+s-1}2^r \sum_{i=r+1}^{r+s}\\
&\qquad\qquad\qquad\qquad\qquad\qquad{}\cdot\sum_{|x_{i}|\leq K} c(x_i,\eta^0)[f_N((\eta^0)_{x_i})- f_N(\eta^0)]+O(N^{-1})\\
&=\frac{1}{r+s}\ \frac{1}{N^{r+s-1}}\big[N-(2K+1)\big]^{r+s-1} s \sum_{|x|\leq K} c(x,\eta^0)[f_N((\eta^0)_{x})- f_N(\eta^0)]\\
&\qquad\qquad\qquad\qquad\qquad\qquad\qquad\qquad\qquad\qquad\qquad\qquad\qquad\qquad{}+O(N^{-1})\\
&=\frac{s}{r+s}\  \sum_{|x|\leq K} c(x,\eta^0)[f_N((\eta^0)_{x})- f_N(\eta^0)]+O(N^{-1})\\
&=\frac{s}{r+s}\psi_N(\Omega_{B}f)(\eta^0)+O(N^{-1}),
\end{align*}
where the first and second equalities require clarification.
 
In the first equality we used 
$$
\sum_{B\subset \{r+1,\ldots, r+s\}}\sum_{i \in B}=\sum_{i=r+1}^{r+s}\sum_{B \subset \{r+1,\ldots, r+s\}:i\in B}
$$
and 
$$
c(x_{j},((\cdots((\cdots(((\eta^{0})^{x_{1},a_1})^{x_{2},a_2})\cdots)^{x_{p},a_p})\cdots)^{x_{r},a_r})_{\{x_{l}: l \in B,l<j\}})=c(x_j,\eta^{0})
$$
with possible exceptions if 
$$
\{x_j-1,x_j,x_j+1\}\cap\bigg[\bigcup_{p\in A^c}\{x_p,x_p-1\}\cup\bigcup_{p\in A}\{x_p,x_p+1\}\cup\bigcup_{p\in B}\{x_p\}\bigg]\ne\varnothing.
$$
That excludes at most $4r+3s$ of the $N$ possible values of $x_j$, hence involves an error of at most $O(N^{-1})$.
In the second equality,
\begin{align*}
&\sum_{|x_m|>K: m\in \{1,2,\ldots,r+s\},m \neq i}\sum_{B\subset \{r+1,\ldots, r+s\}:i\in B}\prod_{j\in B^c}\big[1-c(x_{j},\eta^{0})\big] \prod_{j\in B-\{i\}}c(x_{j},\eta^{0})\nonumber\\
&=\sum_{|x_m|>K: m\in \{1,2,\ldots,r+s\},m \neq i}\prod_{j\in\{r+1,\ldots,r+s\}-\{i\}}\big[1-c(x_{j},\eta^{0})+c(x_{j},\eta^{0})\big]\\
&=[N-(2K+1)]^{r+s-1}.
\end{align*} 

Therefore, we conclude that
\begin{equation} \label{Omega^N}
(\Omega^N_{[r, s]'}\psi_N f)(\eta^0)=\psi_{N}\bigg(\frac{r}{r+s}\Omega_{A'}f+\frac{s}{r+s}\Omega_{B}f\bigg)(\eta^0)+O(N^{-1}),
\end{equation} as desired.  

Since \eqref{Omega^N} holds, uniformly over $\Sigma_N,$ the unique stationary distribution $\pi^N$ of $\bm P_{A'}^r \bm P_{B}^s$ converges weakly  to the unique stationary distribution $\pi^{r/(r+s)}$ of the interacting particle system with generator $\Omega_{C'},$   provided that ergodicity holds for the limiting interacting particle system. Here,  $\Omega_{C'}= \gamma \Omega_{A'}+(1-\gamma) \Omega_B$ with $\gamma:=r/(r+s)$,  specifically 
\begin{align}\label{Omega_C'2}
(\Omega_{C'} f)(\eta)&=\gamma\sum_x c'(x,\eta)[f(\eta_x)-f(\eta)]\nonumber\\
&\quad+\gamma\sum_x [f(_x \eta_{x+1})-f(\eta)]+(1-\gamma)\sum_x c(x,\eta)[f(\eta_x)-f(\eta)].
\end{align}
Our aim is to prove the following theorem.

\begin{theorem}\label{periodic-limit}
Fix integers $r,s\ge1$ and put $\gamma:=r/(r+s)$.  Assume that the interacting particle system on ${\bf Z}$ with generator $\Omega_{C'}$ as in \eqref{Omega_C'2} is ergodic with unique stationary distribution $\pi^\gamma$.  Then $\lim_{N\to\infty}\mu_{[r,s]'}^N=\mu_{(\gamma,1-\gamma)'}$, where $\mu_{(\gamma,1-\gamma)'}$ is as in \eqref{mean-limit}.
\end{theorem}

\begin{proof}
The mean profit per turn to the ensemble of $N$ players playing the nonrandom periodic pattern $(A')^rB^s$ is
\begin{equation} \label{mu^N}
\mu_{[r,s]'}^N =\frac{1}{r+s}\sum_{v=0}^{s-1}\sum_{\eta \in \Sigma_{N}}(\bm\pi^N\bm P^r_{A'}\bm P^v_{B})(\eta)\frac{1}{N}\sum_z[2p_{m_z(\eta)}-1].
\end{equation}
The sum over $\eta$ in \eqref{mu^N} can be expressed, using \eqref{P_A^rP_B^s}, as
\begin{align*}
&\!\!\!\!\!\sum_{\eta^0,\eta}\pi^N(\eta^0)(\bm P^r_{A'}\bm P^v_{B})(\eta^0, \eta)\frac{1}{N}\sum_{z}[2p_{m_z(\eta)}-1]\\
&=\frac{1}{2^rN^{r+v}}\ \sum_{\eta^0}\pi^N(\eta^0)\sum_{A\subset \{1,\ldots, r\}}\sum_{B\subset \{r+1,\ldots, r+v\}}\sum_{x_{i}: i\in \{1,2,\ldots, r+v\}}\\
&\quad\cdot \prod_{j\in B^c}\big[1-c(x_{j},((\cdots((\cdots(((\eta^{0})^{x_{1},a_1})^{x_{2},a_2})\cdots)^{x_{p},a_p})\cdots)^{x_{r},a_r})_{\{x_{l}: l \in B,l<j\}})\big]\\
&\quad\cdot \prod_{j\in B}c(x_{j},((\cdots((\cdots(((\eta^{0})^{x_{1},a_1})^{x_{2},a_2})\cdots)^{x_{p},a_p})\cdots)^{x_{r},a_r})_{\{x_{l}: l \in B,l<j\}})\\
&\quad\cdot\frac{1}{N}\sum_z[2p_{m_z(((\cdots((\cdots(((\eta^{0})^{x_{1},a_1})^{x_{2},a_2})\cdots)^{x_{p},a_p})\cdots)^{x_{r},a_r})_{\{x_{l}: l \in B\}})}-1]\\
&=\frac{1}{2^rN^{r+v}}\ \sum_{\eta^0} \pi^N(\eta^0)\sum_{A\subset \{1,\ldots, r\}}\sum_{B\subset \{r+1,\ldots, r+v\}}\sum_{x_{i}: i\in \{1,2,\ldots, r+v\}}\\
&\quad\cdot \prod_{j\in B^c}\big[1-c(x_{j},\eta^0)\big] \prod_{j\in B}c(x_{j},\eta^0)\frac{1}{N}\sum_z[2p_{m_z(\eta^0)}-1]+O(N^{-1})\\
&=\frac{1}{2^rN^{r+v}}\sum_{\eta^0} \pi^N(\eta^0)\ (2^r) (1)[N-(2k+1)]^{r+v}\frac{1}{N}\sum_z[2p_{m_z(\eta^0)}-1]+O(N^{-1})\\
&=\frac{1}{N}\sum_{z=1}^{N}\sum_{\eta^0} \pi^N(\eta^0)[p_{m_z(\eta^0)}-q_{m_z(\eta^0)}]+O(N^{-1})\\
&=\sum_{k=0}^{1} \sum_{l=0}^{1}(\pi^N)_{-1,1}(k, l)[2p_{2k+l}-1]+O(N^{-1})\\
&=\sum_{k=0}^{1} \sum_{l=0}^{1}(\pi^{r/(r+s)})_{-1,1}(k, l)[2p_{2k+l}-1]+o(1).
\end{align*}
So we have, with $\gamma:=r/(r+s)$,
$$
\mu^{N}_{[r,s]'}\rightarrow (1-\gamma)\sum_{k=0}^{1} \sum_{l=0}^{1}(\pi^\gamma)_{-1,1}(k, l)[2p_{2k+l}-1]=\mu_{(\gamma,1-\gamma)'},
$$ 
as required.
\end{proof}

\newpage
\begin{newreferences}

\item Choi, S. C. (2020) Spatial Parrondo games with spatially dependent game $A$, \url{http://arxiv.org/abs/2101.01172}.

\item Ethier, S. N. and  Kurtz, T. G. (1986) {\em Markov Processes: Characterization and Convergence}, John Wiley \& Sons, New York.

\item Ethier, S. N. and Lee, J. (2009) Limit theorems for Parrondo’s paradox, \textit{Electron. J. Probab.} \textbf{14} 1827--1862.

\item Ethier, S. N. and Lee, J. (2012a) Parrondo games with spatial dependence, \textit{Fluct. Noise Lett.} \textbf{11} 1250004.

\item Ethier, S. N. and Lee, J. (2012b) Parrondo games with spatial dependence, II, \textit{Fluct. Noise Lett.} \textbf{11} 1250030.

\item Ethier, S. N. and Lee, J. (2013a) Parrondo games with spatial dependence and a related spin system, \textit{Markov Process. Relat. Fields} \textbf{19} 163--194.

\item Ethier, S. N. and Lee, J. (2013b) Parrondo games with spatial dependence and a related spin system, II, \textit{Markov Process. Relat. Fields} \textbf{19} 667--692.

\item Ethier, S. N. and Lee, J. (2015) Parrondo games with spatial dependence, III, \textit{Fluct. Noise Lett.} \textbf{14} 1550039.

\item  Liggett, T. M. (1985) {\em Interacting Particle Systems}, Springer-Verlag, New York.

\item Mihailovi\'c, Z. and Rajkovi\'c, M. (2003) One dimensional asynchronous cooperative Parrondo’s games, \textit{Fluct. Noise Lett.} \textbf{3} L389--L398.

\item Toral, R. (2001) Cooperative Parrondo games, \textit{Fluct. Noise Lett.} \textbf{1} L7--L12.

\item Xie, N.-G., Chen, Y., Ye, Y., Xu, G., Wang, L.-G., and Wang, C. (2011) Theoretical analysis and numerical simulation of Parrondo's paradox game in space, \textit{Chaos Solitons Fractals} \textbf{44} 401--414.

\end{newreferences}
\end{document}